\documentclass[12pt]{amsart}
\title{Derivations and gt-henselian field topologies}
\thanks{This research was funded in part by the Austrian Science Fund (FWF) 10.55776/PAT1673125.}
\author{Erik Walsberg}
\email{erik.walsberg@gmail.com}
\usepackage{amsmath, amssymb, amsthm, enumitem, comment}    % need for subequations
\usepackage[mathscr]{euscript}
\usepackage[Symbol]{upgreek}
\usepackage{fullpage} 	% without this, will have wide math-article-paper margins
\usepackage{hyperref}
\usepackage{lineno}
\usepackage[all]{xy}
\usepackage{centernot}
\usepackage{tikz-cd}
\usepackage{xspace}

\usepackage{xcolor}

\DeclareFontFamily{U}{fsy}{}
\DeclareFontShape{U}{fsy}{m}{n}{<->s*[.9]psyr}{}
\DeclareSymbolFont{der@m}{U}{fsy}{m}{n}
\DeclareMathSymbol{\der}{\mathord}{der@m}{182}

\DeclareFontFamily{U}{BOONDOX-calo}{\skewchar\font=45 }
\DeclareFontShape{U}{BOONDOX-calo}{m}{n}{
  <-> s*[1.05] BOONDOX-r-calo}{}
\DeclareFontShape{U}{BOONDOX-calo}{b}{n}{
  <-> s*[1.05] BOONDOX-b-calo}{}
\DeclareMathAlphabet{\mathcalboondox}{U}{BOONDOX-calo}{m}{n}
\SetMathAlphabet{\mathcalboondox}{bold}{U}{BOONDOX-calo}{b}{n}
\DeclareMathAlphabet{\mathbcalboondox}{U}{BOONDOX-calo}{b}{n}

\DeclareSymbolFont{imag@m}{OT1}{cmr}{m}{ui}
\DeclareMathSymbol{\imag}{\mathord}{imag@m}{105}

\DeclareMathOperator*{\forkindep}{\raise0.2ex\hbox{\ooalign{\hidewidth$\vert$\hidewidth\cr\raise-0.9ex\hbox{$\smile$}}}}

\newcommand{\Sa}[1]{\ensuremath{\mathscr{#1}}}

\newcommand{\lowenheim}{L\"owenheim-Skolem\xspace}

\newcommand{\mfrak}{\mathfrak{m}}

\newcommand{\Spec}{\operatorname{Spec}}

\newcommand{\Frac}{\operatorname{Frac}}

\newcommand{\pac}{\mathrm{PAC}}

\newtheorem*{claim-star}{Claim}
\newtheorem{theorem}{Theorem}[section] % numbered like the section
\newtheorem{lemma}[theorem]{Lemma}

\newtheorem{prop-def}[theorem]{Proposition-Definition}
\newtheorem{corollary}[theorem]{Corollary}
\newtheorem{fact}[theorem]{Fact}
\newtheorem{fact-eh}[theorem]{Fact(?)}

\newtheorem{conjecture}[theorem]{Conjecture}

\newtheorem{question}[theorem]{Question}
\newtheorem{proposition}[theorem]{Proposition}
\newtheorem{proposition-eh}[theorem]{Proposition(?)}
\newtheorem*{theorem-star}{Theorem}
\newtheorem*{conjecture-star}{Conjecture}
\newtheorem*{lemma-star}{Lemma}

\newtheorem*{lemma6.2alt}{Lemma \ref*{confusing-v2}$'$}

\theoremstyle{definition}

\theoremstyle{remark}

\newcommand{\A}{\mathbb{A}}

\newcommand{\Q}{\mathbb{Q}}
\newcommand{\R}{\mathbb{R}}

\newcommand{\N}{\mathbb{N}}

\begin{document}

\maketitle

\begin{abstract}
Suppose that $K$ is a characteristic zero field with infinite transcendence degree over its prime subfield.
We show that if there is a gt-henselian topology on $K$  then there are $2^{2^{|K|}}$ pairwise incomparable gt-henselian topologies on $K$.
It follows by applying a recent theorem of Will Johnson that if $K$ is large and countable then there are $2^{2^{\aleph_0}}$ pairwise incomparable gt-henselian topologies on $K$.
We also formulate several conjectures concerning gt-henselian field topologies and their relationship with the \'etale-open topology.
\end{abstract}

\section{Introduction}
Let $K$ be a field.
All field topologies are assumed to be Hausdorff.
Recall that a \textbf{gt-henselian topology} on $K$ is a non-discrete field topology such that for every neighborhood $U$ of $-1$ and $d \in \N$ there is a neighborhood $V$ of $0$ such that $x^{d + 2} + x^{d+1} + a_d x^d + \cdots + a_1 x + a_0$ has a simple root in $U$ when $a_0,\ldots,a_d \in V$.
A field topology is gt-henselian if and only if it is not discrete and satisfies polynomial inverse and implicit function theorems~\cite[Prop.~6.2]{tbnidatf}.
If $K$ is the fraction field of a henselian local domain $R$ then the $R$-adic topology (recalled below) is a gt-henselian field topology.
Recall that $K$ is \textbf{large} if any smooth $1$-dimensional $K$-variety with a $K$-point has infinitely many $K$-points.
See \cite{Pop-little, open-problems-ample, firstpaper} for background on and examples of this important class of fields.
Any field admitting a gt-henselian topology is large~\cite[Cor.~8.15]{field-top-2}.
Johnson proved a partial converse~\cite{lagth}.

\begin{fact}\label{fact:will}
Suppose that $K$ is large and countable.
Then $K$ admits a gt-henselian field topology.
Furthermore the intersection of all topologies on $V(K)$ induced by gt-henselian field topologies agrees with the \'etale-open topology on $V(K)$ for any $K$-variety $V$.
\end{fact}

We recall the \'etale-open topology below.
We first state our main results.

\begin{theorem}
Suppose that $K$ is characteristic zero and has infinite transcendence degree.
\begin{enumerate}[leftmargin=*]
\item If there is a gt-henselian topology on $K$ then there are $2^{2^{|K|}}$ pairwise incomparable gt-henselian topologies on $K$ and $2^{|K|}$ pairwise incomparable gt-henselian topologies of weight $\le |K|$ on $K$.
\item If $\uptau$ is a gt-henselian topology  of weight $\le |K|$ on $K$ then there are $2^{2^{|K|}}$ pairwise incomparable gt-henselian topologies on $K$ refining $\uptau$.
\item If $\uptau$ is a locally bounded gt-henselian topology on $K$ then there are $2^{|K|}$ pairwise incomparable locally bounded gt-henselian topologies on $K$ refining $\uptau$.
\item If $K$ is large and countable then there are $2^{2^{\aleph_0}}$ pairwise incomparable gt-henselian topologies on $K$ and $2^{\aleph_0}$ pairwise incomparable second countable gt-henselian topologies on $K$.
\end{enumerate}
\end{theorem}

This seems to be bad news for the theory of gt-henselian field topologies.
There are too many, their structure is too complex, and they are too non-canonical.
However, we are really interested in the relationship between gt-henselian topologies and the \'etale-open topology. 
We formulate some conjectures about this in Section~\ref{section:qc}.

\section{Background}
Throughout $K$ is a field and the transcendence degree of $K$ is the transcendence degree over the prime subfield.
Let $V$ range over $K$-varieties and let $V(K)$ be the set of $K$-points of $V$.
We let $\A^n = \Spec K[x_1, \ldots, x_n]$ and recall that $\A^n(K)$ is canonically identified with $K^n$ for each $n \ge 1$.
An \textbf{E-subset} of $V(K)$ is a set of the form $f(W(K))$ for $f\colon W \to V$ an \'etale morphism of $K$-varieties.
The \textbf{\'etale-open} (or $\Sa E_K$-) topology on $V(K)$ is the topology with basis the collection of E-sets. 
The $\Sa E_K$-topology on $K = \A^1(K)$ is discrete if and only if the $\Sa E_K$-topology on $V(K)$ is discrete for every $V$ if and only if $K$ is not large~\cite{firstpaper}.
We say that the $\Sa E_K$-topology \textbf{is a field topology} if there is a field topology $\uptau$ on $K$ such that the $\uptau$-topology on $V(K)$ agrees with the $\Sa E_K$-topology for every $V$ and we say that a field topology $\uptau$ on $K$ refines the $\Sa E_K$-topology when the $\uptau$-topology on $V(K)$ refines the $\Sa E_K$-topology for every $V$.
By \cite[Prop.~4.9]{firstpaper} the $\Sa E_K$-topology is a field topology if and only if the $\Sa E_K$-topology on $K^n$ agrees with the product topology given by the $\Sa E_K$-topology on $K$ for every $n \ge 2$.

\medskip
The \textbf{weight} $w(\uptau)$ of a topology $\uptau$ is the minimal cardinality of a basis.
There are $2^{|X|}$ topologies of weight $\le |X|$ on an infinite set $X$.
Let $\uptau$ be a field topology on $K$.
Then $w(\uptau)\le |K| + \eta$ for $\eta$ the minimal cardinality of a neighborhood basis at zero.
Recall that $X \subseteq K$ is \textbf{bounded} if for any neighborhood $U$ of $0$ we have $\lambda X \subseteq U$ for some $\lambda \in K^\times$.
If $X$ is bounded then for any neighborhood $U$ of $0$ we have $VX \subseteq U$ for some neighborhood $V$ of $0$~\cite[Lemma~2.1]{Prestel1978}.
Bounded sets are closed under affine transformations.
If $U$ is a nonempty bounded open set then $\{ aU + b : a \in K^\times, b \in K\}$ is a basis for $\uptau$; $\uptau$ is \textbf{locally bounded} if some nonempty open set is bounded.
Hence $w(\uptau) \le |K|$ when $\uptau$ is locally bounded.
% and so there are $\le 2^{|K|}$ locally bounded field topologies on $K$.

\medskip
All rings are commutative with unit.
We let $\Frac(R)$ be the fraction field of a domain $R$.
If $R$ is a local domain with $K = \Frac(R)$ then the \textbf{$R$-adic topology} is the topology on $K$ with basis $\{aR + b : a \in K^\times, b \in K\}$, this is a locally bounded field topology~\cite[Thm.~2.2]{Prestel1978}.
Furthermore, $\uptau$ is a V-topology if $(K\setminus U)^{-1}$ is bounded for any neighborhood $U$ of $0$.
% A field topology is a \textbf{V-topology} if and only if it is induced by a valuation or absolute value on $K$.
% If $R$ is a noetherian local domain of Krull dimension $\ge 2$ then the $R$-adic topology on the fraction field of $R$ is not locally bounded~\cite[Prop.~2.1]{field-top-1}.
A \textbf{t-henselian topology} on $K$ is a gt-henselian V-topology~\cite[Prop.~8.3]{field-top-2}.
The topology induced by a non-trivial henselian valuation and the order topology over a real closed field are both t-henselian.
T-henselianity was introduced by Prestel and Ziegler who showed that if $K$ is not separably closed then $K$ admits a t-henselian topology if and only if $K$ admits a unique t-henselian topology if and only if $K$ is elementarily equivalent to a field which admits a non-trivial henselian valuation~\cite{Prestel1978}.

\section{Questions and conjectures}\label{section:qc}
% Hence the $\Sa E_K$-topology is a field topology if and only if we have the following: If $f \colon V \to \A^n$ is an \'etale morphism of $K$-varieties and $p$ is a $K$-point of $V$ then there is an \'etale morphism $f_i \colon V_i \to \A^1$ of $K$-varieties and a $K$-point $p_i$ of $V_i$ for $i = 1,\ldots,n$ such that $f(p) = (f_1(p_1),\ldots,f_n(p_n))$ and $f_1(V_1(K)) \times \cdots \times f_n(V_n(K))$ is contained in $f(V(K))$.

\begin{question}
For which large fields $K$ is the $\Sa E_K$-topology a field topology?
\end{question}

We know that the $\Sa E_K$-topology is a field topology in the following cases:
\begin{enumerate}[leftmargin=*]
\item When $K$ is not separably closed and t-henselian. 
In this case the $\Sa E_K$-topology agrees with the unique t-henselian topology over $K$~\cite[Thm.~B]{firstpaper}.
Furthermore the $\Sa E_K$-topology is a V-topology if and only if $K$ is t-henselian and not separably closed.
\item When $K$ is the fraction field of a quasi-excellent henselian local ring $R$.
In this case the $\Sa E_K$-topology agrees with the $R$-adic topology~\cite[Cor.~1.4]{field-top-2}.
% This covers fractions fields of complete noetherian local rings such that $K[[t_1,\ldots,t_n]]$.
\end{enumerate}

The $\Sa E_K$-topology is not a field topology for $\pac$ $K$~\cite[Prop.~7.1]{field-top-2}.
% and more generally when $K$ is $\prc$ and not real closed [??].
% % In the particular the \'etale-open topology over a pseudofinite field or the maximal totally real extension of $\Q$ is not a field topology.
% Fact~\ref{fact:basic gt} was the motivation for gt-henselian topologies [??].
Fact~\ref{fact:basic gt} is \cite[Lemma~8.14]{field-top-2}.

\begin{fact}\label{fact:basic gt}
Any gt-henselian topology on $K$ refines the $\Sa E_K$-topology.
If $K$ is large and the $\Sa E_K$-topology is a field topology then the $\Sa E_K$-topology is gt-henselian.
\end{fact}

Hence there is a coarsest gt-henselian topology on $K$ when the $\Sa E_K$-topology is a field topology.
Johnson's theorem on gt-henselian topologies gives the following partial converse.

\begin{fact}\label{cor:will}
If $K$ is countable and large then the $\Sa E_K$-topology is a field topology if and only if there is a coarsest gt-henselian topology on $K$.
\end{fact}

We conjecture that this holds in general.

\begin{conjecture}
If $K$ is large then the $\Sa E_K$-topology is the intersection of all gt-henselian topologies on $K$ and hence the $\Sa E_K$-topology is a field topology if and only if there is a coarsest gt-henselian topology on $K$.
\end{conjecture}

By \cite[Thm.~A]{large->henselian} any $\aleph_1$-saturated large field is the fraction field of a henselian local domain.

\begin{conjecture}\label{conj:sat}
Suppose that $K$ is large and $\lambda$-saturated for a sufficiently large cardinal $\lambda$.
Then the $\Sa E_K$-topology is the intersection of all $R$-adic topologies for $R$ ranging over henselian local domains $R \subseteq K$ with $\Frac(R) = K$.
Furthermore the $\Sa E_K$-topology is a field topology if and only if there is a henselian local domain $R\subseteq K$ with $\Frac(R) = K$ such that for any other henselian local domain $S \subseteq K$ with $\Frac(S) = K$ we have $aS \subseteq R$ for some $a \in K^\times$.
\end{conjecture}

We can prove a weak version of Conjecture~\ref{conj:sat} by applying Johnson's theorem.

\begin{corollary}\label{cor:sat}
Suppose that $K$ is large and highly saturated and let $X \subseteq K^n$ be definable.
Then $X$ is $\Sa E_K$-open if and only if $X$ is open in the $R$-adic topology for every henselian local domain $R \subseteq K$ with $\Frac(R) = K$.
\end{corollary}

\begin{proof}
The left to right direction holds by Fact~\ref{fact:basic gt}.
We suppose that $X$ is not $\Sa E_K$-closed and produce a henselian local $R \subseteq K$ such that $\Frac(R) = K$ and $X$ is not closed in the $R$-adic topology.
Suppose that $p \in K^n \setminus X$ is in the $\Sa E_K$-closure of $X$.
After possibly translating we may suppose that $p$ is the origin.
So every E-subset of $K^n$ containing the origin intersects $X$.
Let $F$ be a countable elementary subfield of $K$ such that $X$ is definable with parameters from $F$.
Let $X^*$ be the subset of $F^n$ defined by any formula with parameters from $F$ which defines $X$.
It follows that every E-subset of $F^n$ containing the origin intersects $X^*$, hence the origin is in the $\Sa E_F$-closure of $X^*$  but not in $X^*$.
By Johnson's theorem there is a gt-henselian topology $\uptau$ on $F$ such that the origin is in the $\uptau$-closure of $X^*$.
Let $\mathcal{B}$ be a neighborhood basis for $\uptau$ at the identity.
By saturation we may suppose that there is a family $\mathcal{B}^*$ of subsets of $K$ such that $(K,\mathcal{B}^*)$ is an elementary extension of $(F,\mathcal{B})$.
Let $\mfrak = \bigcap \mathcal{B}^*$ and let $R = F + \mfrak$.
By the proof of \cite[Prop.~8.12]{field-top-2} $R$ is a henselian local domain with maximal ideal $\mfrak$ and $K = \Frac(R)$.
We show that the origin is in the $R$-adic closure of $X$, equivalently that $b\mfrak$ intersects $X$ for every $b \in K^\times$.
By saturation it is enough to fix $U \in \mathcal{B}$, let $U^*$ be the corresponding element of $\mathcal{B}^*$, and show that $bU^*$ intersects $X$ for every $b \in K^\times$.
This follows by elementary transfer as $bU$ intersects $X^*$ for every $b \in F^\times$.
\end{proof}

E-sets are existentially definable.
Hence if the $\Sa E_K$-topology is a field topology then there is a gt-henselian topology on $K$ which admits a basis consisting of existentially definable sets.
% We also conjecture that a converse holds.

\begin{conjecture}\label{conj:ez gen}
Suppose that there is a gt-henselian topology on $K$ which admits a basis consisting of definable sets.
Then the $\Sa E_K$-topology is a field topology.
\end{conjecture}

Conjecture~\ref{conj:ez gen} holds for a class of large fields which contains the main examples of logically tame fields.
We say that $K$ is \textbf{\'ez} if $K$ is large and every definable subset of every $K^n$ is a finite union of definable $\Sa E_K$-open subsets of Zariski closed sets.
\'Ez fields are perfect and it is an informal conjecture that all known logically tame perfect large fields, and in particular all logically tame perfect fields known before \cite{cxf-paper}, are \'ez.
See \cite{secondpaper} for background.

\begin{proposition}\label{prop:ez}
Suppose that $\uptau$ is a non-discrete field topology on $K$ and one of the following holds.
\begin{enumerate}[leftmargin=*]
\item $K$ is \'ez and there is a basis for $\uptau$ consisting of definable sets.
\item $K$ is perfect and there is a basis for $\uptau$ consisting of existentially definable sets.
\end{enumerate}
Then the $\Sa E_K$-topology on $V(K)$ refines the $\uptau$-topology for every $K$-variety $V$.
Hence if $\uptau$ is additionally gt-henselian then $\uptau$ agrees with the $\Sa E_K$-topology.
\end{proposition}

\begin{proof}
The second claim follows from the first by Fact~\ref{fact:basic gt}.
We prove the first claim.
By \cite[Lemma~4.8]{firstpaper} it suffices to show that the $\Sa E_K$-topology on $K$ refines $\uptau$.
If $K$ is \'ez then any definable subset of $K$ is the union of an $\Sa E_K$-open set and a finite set.
If $K$ is perfect then any existentially definable subset of $K$ is the union of an $\Sa E_K$-open set and a finite set~\cite[Cor.~A]{secondpaper}.
Hence we may suppose that there is a neighborhood basis $\mathcal{B}$ of $\uptau$ at $0$ such that every $U \in \mathcal{B}$ is the union of an non-empty $\Sa E_K$-open set and a finite set.
It suffices to fix $U \in \mathcal{B}$ and produce $\Sa E_K$-open $V$ such that $0 \in V \subseteq U$.
Fix $V^* \in \mathcal{B}$ such that $V^* - V^* \subseteq U$.
We have $V^* = O\cup A$ for finite $A$ and nonempty $\Sa E_K$-open $O \subseteq K$.
Note that $O$ is also $\uptau$ open as $A$ is $\uptau$-closed.
Take $V = O - \beta$ for any $\beta \in O$.
\end{proof}

% Fact~\ref{fact:ez prod} is \cite[Lemma~4.8]{firstpaper}.

% \begin{fact}\label{fact:ez prod}
% Let $\uptau$ be a field topology on $K$.
% If the $\Sa E_K$-topology on $K$ refines $\uptau$ then the $\Sa E_K$-topology on $V(K)$ refines the $\uptau$-topology for any $K$-variety $V$.
% \end{fact}

\section{Derivations and gt-henselian topologies}
We suppose the following throughout this section:
\begin{enumerate}[leftmargin=*]
\item $K$ is characteristic zero,
\item $\uptau$ is a field topology on $K$,
\item and $I$ is a set of derivations $K \to K$.
\end{enumerate}
Let $\sigma_I$ be the embedding $K \to K \times K^I$ given by $\sigma_I(a) = (a,(\der(a))_{\der \in I}))$.
Equip $K \times K^I$ with the product topology induced by $\uptau$ and let $\uptau_I$ be the topology induced on $K$ by $\sigma_I$.
If $I = \{\der\}$ we let $\uptau_\der$ be $\uptau_I$.
Now $\uptau_I$ is the topology with basis  sets of the form $$\{ a \in K : a \in U, \der_{1}(a) \in V_1,\ldots,\der_{n}(a) \in V_n\} = U \cap \der^{-1}_{1}(V_1)\cap \cdots \cap \der^{-1}_{n}(V_n)$$ for $\uptau$-open $U,V_1,\ldots,V_n \subseteq K$ and distinct $\der_1,\ldots,\der_n\in I$.
Note that $$\{ (a_1,\ldots,a_n) \in K^n : (a_{j_1},\ldots,a_{j_\ell},\der_{1}(a_{i_1}),\ldots,\der_{m}(a_{i_m})) \in W\}$$ is $\uptau_\der$-open for any $\uptau$-open $W \subseteq K^{\ell + m}$, $\der_1,\ldots,\der_m\in I$, and $j_1,\ldots,j_\ell,i_1,\ldots,i_m \in \{1,\ldots,n\}$.

\begin{lemma}\label{lem:der 0}
The topology $\uptau_I$ is a field topology refining $\uptau$, we have $w(\uptau_I)\le |I| + w(\uptau)$, and $\uptau_I$ strictly refines $\uptau$ if and only if some $\der \in I$ is not continuous.
\end{lemma}

\begin{proof}
It is clear from the definition that $\uptau_I$ refines $\uptau$.
If each $\der \in I$ is continuous then $\sigma_I$ is continuous hence $\uptau_I = \uptau$.
If $\der \in I$ is not continuous then $\der^{-1}(U)$ is not $\uptau$-open for some $\uptau$-open $U \subseteq K$, hence $\uptau_I$ strictly refines $\uptau$.
The weight inequality follows as $w(\uptau_I)$ is bounded above by the weight of the product topology on $K \times K^I$, and this is $|I| + w(\uptau)$.

\medskip
We show that $\uptau_I$ is a field topology.
Consider $K \times K^I$ to be a group with pointwise addition.
Then $\sigma_I$ embeds the additive group of $K$ into $K\times K^I$, so $\uptau_I$ is an additive group topology.
We show that multiplication is $\uptau_I$-continuous.
Fix $\uptau$-open $U,V_1,\ldots,V_n \subseteq K$ and $\der_1,\ldots,\der_n \in I$.
Set $O = U \cap \der^{-1}_{1}(V_1)\cap \cdots \cap \der^{-1}_{n}(V_n)$.
We show that the set of $(a,b) \in K^2$ with $ab \in O$ is $\uptau_I$-open.
For any $a,b \in K$ we have $ab \in O$ if and only if $ab \in U$ and $a\der_i(b) + b\der_i(a) \in V_i$ for $i = 1,\ldots,n$.
Let $W$ be the set of $(a,b,a'_1,b'_1,\ldots,a'_n,b'_n) \in K^{2(n+1)}$ such that $ab \in U$ and $ab'_i + ba'_i \in V_i$ for $i = 1,\ldots,n$.
So $W$ is $\uptau$-open.
Then for any $a,b \in K$ we have $ab \in O$ if and only if $(a,b,\der_1(a),\der_1(b),\ldots,\der_n(a),\der_n(b)) \in W$.

% \medskip
% We now fix a $\uptau$-neighborhood $U$ of $0$ and produce a $\uptau$-neighborhood $V$ of $0$ such that $\sigma(a),\sigma(b) \in V\times V$ implies $\sigma(ab) \in U \times U$.
% We have $$ \sigma(ab) = (ab, a\der(b) + b\der(a)).$$
% Now $(a,b,a',b') \mapsto ab' + ba'$ is a $\uptau$-continuous map $K^4 \to K$ that sends $(0,0,0,0)$ to $0$, so there is a $\uptau$-neighborhood $V$ of $0$ such that $ab' + ba' \in U$ when $a,b,a',b'\in V$.
% After possibly shrinking $V$ we may suppose that $a,b \in V$ implies $ab \in U$.
% Hence if $\sigma(a), \sigma(b) \in V \times V$ then $a,b,\der(a),\der(b) \in V$, hence $\sigma(ab) \in U \times U$.

% \medskip
% We now need to show that for any $\lambda \in K$ and $\uptau$-neighborhood $U$ of $0$ there is an $\uptau$-neighborhood $V$ of $0$ such that $\sigma(a) \in V \times V$ implies $\sigma(\lambda a) \in U$.
% Again, $(a,b,a',b') \to ab' + ba'$ is a $\uptau$-continuous map $K^4 \to K$ that sends $(\lambda, 0, \der(\lambda),0)$ to $0$.

\medskip
It remains to show that multiplicative inverse is a $\uptau_I$-continuous function $K^\times \to K^\times$.
Let $U,V_1,\ldots,V_n$, $\der_1,\ldots,\der_n$, and  $O$ be as above.
We show that the set of $a \in K^\times$ such that $1/a \in O$ is $\uptau_I$-open.
We have 
$$ \der_i(1/a) = -\frac{\der_i(a)}{a^2}\quad\text{for  } i = 1,\ldots,n.$$
Now $h(a,a'_1,\ldots,a'_n) = (a,-a'_1/a^2,\ldots,-a'_n/a^2)$ gives a $\uptau$-continuous map $(K^\times)^n \to (K^\times)^n$.
Let $W$ be the set of $(a,a'_1,\ldots,a'_n) \in (K^\times)^n$ such that $h(a,a'_1,\ldots,a'_n)\in U \times V_1\times\cdots\times V_n$.
Then $W$ is $\uptau$-open and $1/a \in O$ if and only if $(a,\der_1(a),\ldots,\der_n(a)) \in W$ for any $a \in K^\times$.
\end{proof}

\begin{lemma}\label{lem:discrete}
The following are equivalent.
\begin{enumerate}[leftmargin=*]
\item $\uptau_I$ is not discrete.
\item $U\cap \der^{-1}_{1}(U)\cap \cdots \cap \der^{-1}_{n}(U) \ne \{0\}$ for every $\uptau$-neighborhood $U$ of $0$ and $\der_1,\ldots,\der_n\in I$.
\end{enumerate}
\end{lemma}

\begin{proof}
Note that $\uptau_I$ is discrete if and only if $0$ is an isolated point.
Note also that sets of the form described in (2) form a neighborhood basis for $\uptau_I$ at $0$.
\end{proof}

\begin{lemma}\label{lem:loc bounded}
Suppose that $\uptau$ is locally bounded and $I$ is finite.
Then $\uptau_I$ is locally bounded.
\end{lemma}

\begin{proof}
We only treat the case when $I = \{\der\}$.
The general case follows by slight modifications of our proof.
If $\uptau_\der$ is discrete then $\uptau_\der$ is trivially locally bounded.
Hence we may suppose that $\uptau_\der$ is not discrete.
Let $U$ be a $\uptau$-bounded open neighborhood of zero.
Set $P = U \cap \der^{-1}(U)$.
We show that $P$ is $\uptau_\der$-bounded.
Let $P^*$ be an arbitrary $\uptau$-open neighborhood of $0$.
It is enough to show that some non-zero multiple of $P$ is contained in $P^* \cap \der^{-1}(P^*)$.
Fix a $\uptau$-neighborhood $Q$ of $0$ such that $Q + Q \subseteq P^*$.
As $U$ is bounded there is a $\uptau$-neighborhood $V$  of $0$ such that $VU \subseteq Q$.
By Lemma~\ref{lem:discrete} there is $\lambda \in V \setminus \{0\}$ such that $\der(\lambda) \in V$.
Then $\lambda U \subseteq Q \subseteq P^*$, hence $\lambda P \subseteq P^*$.
Furthermore if $a \in P$ then we have $\der(\lambda a) = \lambda \der(a) + a \der(\lambda) \in Q + Q \subseteq P^*$, hence $\lambda a \in \der^{-1}(P^*)$.
Hence $\lambda P \subseteq P^* \cap \der^{-1}(P^*)$.
\end{proof}

\begin{proposition}\label{prop:key}
Suppose that $\uptau$ is gt-henselian and $\uptau_I$ is  not discrete.
Then $\uptau_I$ is also gt-henselian.
\end{proposition}

We need $\uptau_I$ to be non-discrete as gt-henselian topologies are non-discrete by definition.

\begin{proof}
Fix $d \ge 1$ and let $\alpha = (\alpha_0,\ldots,\alpha_d)$ range over $K^{d + 1}$.
Declare
$$ p_\alpha(x) = x^{d + 2} + x^{d+1} + \alpha_d x^d + \cdots + \alpha_1 x + \alpha_0 \in K[x].$$
Fix a $\uptau$-neighborhood $U$ of $-1$, a $\uptau$-neighborhood $V$ of $0$, and $\der_1,\ldots,\der_n\in I$.
We produce a $\uptau$-neighborhood $O$ of $0$ such that if $\alpha_i \in O \cap \der^{-1}_1(O) \cap \cdots \cap \der^{-1}_n(O)$ for $i = 0,\ldots,d$, then there is a simple root $\beta$ of $p_\alpha$ in $U\cap \der^{-1}_{1}(V)\cap\cdots\cap\der^{-1}_{n}(V)$.
This shows that $\uptau_I$ is gt-henselian.

\medskip
Suppose that $g$ is a polynomial in $K[x]$ and let $i$ range over $\{1,\ldots,n\}$.
Let $\der_i g$ be given by applying $\der_i$ to each coefficient of $g$.
If $\beta \in K$ is a root of $g$ then we have
\[
0 = \der_i(0) = \der_i(g(\beta)) = (\der_i g)(\beta) + g'(\beta) \der_i(\beta).
\]
Hence if $\beta$ is additionally simple then
\[
\der_i(\beta) = -\frac{(\der_i g)(\beta)}{g'(\beta)}.
\]
It follows that if $\beta\in K$ is a simple root of $p_\alpha$ then
\[
\der_i(\beta) = -\frac{\der_i(\alpha_d) \beta^d + \cdots + \der_i(\alpha_1) \beta + \der_i(\alpha_0)}{(d + 2) \beta^{d+1} + (d+1)\beta^d + d \alpha_d \beta^{d - 1} + \cdots + 2\alpha_2 \beta + \alpha_1}.
\]
Let $W$ be the set of $(a_0,\ldots,a_d,b, a'_0,\ldots,a'_d) \in K^{2d + 3}$ such that
$$ (d + 2) b^{d+1} + (d+1)b^d + d a_d b^{d - 1} + \cdots + 2a_2 b + a_1 \ne 0.$$
So $W$ is $\uptau$-open neighborhood of $(0,\ldots,0,-1,0,\ldots,0)$.
Let $h \colon W \to K$ be given by
$$h(a_0,\ldots,a_d,b, a'_0,\ldots,a'_d) = -\frac{a'_d b^d + \cdots + a'_1 b + a'_0}{(d + 2) b^{d+1} + (d+1)b^d + d a_d b^{d - 1} + \cdots + 2a_2 b + a_1}.$$
Then $h$ is $\uptau$-continuous and $h(0,\ldots,0,-1,0,\ldots,0) = 0$.
Hence there are $\uptau$-neighborhoods $Q,O$ of $-1,0$, respectively, such that $h(a_0,\ldots,a_d,b, a'_0,\ldots,a'_d)$ is in $W$ when each $a_i,a'_i$ is in $O$ and $b \in Q$.
After replacing $U$ with $U \cap Q$ we suppose that $U \subseteq Q$.
After possibly shrinking $O$ we may suppose by gt-henselianity of $\uptau$ that $p_\alpha$ has a simple root in $U$ when $\alpha_0,\ldots,\alpha_d \in O$.
Hence if $\alpha_0,\ldots,\alpha_d,\der_1(\alpha_0),\ldots,\der_1(\alpha_d),\ldots,\der_n(\alpha_0),\ldots,\der_n(\alpha_d)\in O$ then there is a simple root $\beta$ of $p_\alpha$ in $U$ such that $\der_i(\beta) \in V$ for $i = 1,\ldots,n$.
\end{proof}

We now recall a basic linear-algebraic fact which we leave to the reader.

\begin{fact}\label{fact:lin alg}
Suppose that $f_1,\ldots,f_n$ are $K$-linearly independent functions $X \to K$ for some set $X$.
Then there are $a_1,\ldots,a_n \in X$ such that the vectors $(f_1(a_i),\ldots,f_n(a_i)) \in K^n$ are $K$-linearly independent for $i = 1,\ldots,n$.
\end{fact}

% \begin{proof}
% Let $W$ be the dual vector space to $K^n$.
% Given $a = (a_1,\ldots,a_m) \in X^m$ let $V(a)$ be the linear subspace of $T \in W$ such that $T(f_1(a_i),\ldots,f_n(a_i)) = 0$ for $i = 1,\ldots,m$.
% We apply induction to show that for every $m = 1,\ldots,n$ there is $a$ such that $\dim V(a) \le n - m$.
% The case $m = 1$ is trivial.
% Suppose that $a = (a_1,\ldots,a_m) \in X^m$ is such that $\dim V(a) \le n - m$.
% If $\dim V(a) = 0$ then we take $a_{m + 1},\ldots,a_n$ to be arbitrary.
% Otherwise fix nonzero $(\lambda_1,\ldots,\lambda_n) \in V(a)$.
% Then we have $\lambda_1 f_1 + \cdots + \lambda_n f_n \ne 0$, so there is $a_{m + 1}$ such that 
% \end{proof}

Derivations $\der_1\ldots,\der_n\colon K \to K$ are linearly independent if they are linearly independent elements of the $K$-vector space of functions $K \to K$.
The \textbf{constant subfield} $\mathrm{Cons}(\der)$ of a derivation $\der\colon K \to K$ is the set of $a \in K$ such that $\der(a) = 0$.
Note that $\der$ is $\mathrm{Cons}(\der)$-linear.

\begin{lemma}\label{lem:two der}
Let $\der_1,\ldots, \der_n$ be linearly independent derivations $K \to K$ with $F = \bigcap_{i = 1}^{n} \mathrm{Cons}(\der_i)$  $\uptau$-dense in $K$.
Then $\{(a, \der_1(a),\ldots, \der_n(a)) : a \in K \}$ is a $\uptau$-dense subset of $K^{n + 1}$.
\end{lemma}

\begin{proof}
By Fact~\ref{fact:lin alg} there are $t_1,\ldots,t_n \in K$ so that the vectors $(\der_1(t_i),\ldots, \der_n(t_i))$ are $K$-linearly independent for $i = 1,\ldots,n$.
Given $\alpha = a_0 + a_1 t_1 + \cdots + a_n t_n$ with $a_0,\ldots,a_n\in F$ we have
\[
\begin{pmatrix}
\alpha \\ \der_1(\alpha) \\ \vdots \\ \der_n(\alpha)
\end{pmatrix}
=
\begin{pmatrix}
a_0 + a_1 t_1 +\cdots + a_n t_n \\ a_1 \der_1(t_1) +\cdots + a_n \der_1(t_n) \\ \vdots \\ a_1 \der_n(t_1) +\cdots + a_n \der_n(t_n)
\end{pmatrix}
\]
We let $T$ be the $K$-linear transformation $K^{n +1} \to K^{n + 1}$ given as follows:
\[
T\begin{pmatrix} x_0 \\ \vdots \\ x_n \end{pmatrix} = \begin{pmatrix}
x_0 + x_1 t_1 +\cdots + x_n t_n \\ x_1 \der_1(t_1) +\cdots + x_n \der_1(t_n) \\ \vdots \\ x_1 \der_n(t_1) +\cdots + x_n \der_n(t_n)
\end{pmatrix}
\]
Note that $T(F^{n+1}) \subseteq \{(a, \der_1(a),\ldots, \der_n(a)) : a \in K \}$, so it is enough to show that $T(F^{n+1})$ is $\uptau$-dense in $K^{n+1}$.
As $F^{n+1}$ is $\uptau$-dense in $K^{n+1}$ and $T$ is linear it is sufficient to note that $T$ is invertible.
This follows as
\[
\det(T) = \det 
\begin{pmatrix}
1 & t_1 & \cdots & t_n \\
0 & \der_1(t_1) & \cdots & \der_n(t_1) \\
\vdots &\vdots &\vdots &\vdots \\
0 & \der_n(t_n) & \cdots & \der_n(t_n) 
\end{pmatrix}
= \det
\begin{pmatrix}
\der_1(t_1) & \cdots & \der_n(t_1) \\
\vdots  &\vdots &\vdots \\
\der_n(t_n) & \cdots & \der_n(t_n) \end{pmatrix}
\ne 0. \qedhere
\]
\end{proof}

\begin{lemma}\label{lem:ind-family}
Suppose that $I$ is $K$-linearly independent and $F = \bigcap_{\der \in I} \mathrm{Cons}(\der)$ is dense.
Let $I_0,I_1$ be subsets of $I$ and let $\uptau_i = \uptau_{I_i}$ for $1 = 1,2$.
Then $\uptau_2$ refines $\uptau_1$ if and only if $I_1 \subseteq I_2$.
Hence $\uptau_1,\uptau_2$ are incomparable if and only if $I_1,I_2$ are incomparable under containment.
\end{lemma}

\begin{proof}
It is clear from the definitions that $\uptau_2$ refines $\uptau_1$ when $I_1 \subseteq I_2$.
Suppose that $\der \in I_1 \setminus I_2$.
Let $U,U^*$ be disjoint nonempty $\uptau$-open subsets of $K$ and let $P = \der^{-1}(U)$.
We show that $P$ is not $\uptau_2$-open.
It suffices to fix distinct $\der_1,\ldots,\der_n \in I_2$ and $\uptau$-open subsets $V,V_1,\ldots,V_n$ of $K$ and show that $V \cap \der^{-1}_1(V_1)\cap\cdots\cap\der^{-1}_n(V_n)$ is not contained in $P$.
By Lemma~\ref{lem:two der} there is $a \in K$ such that $(a,\der_1(a),\ldots,\der_n(a),\der(a)) \in V \times V_1 \times \cdots\times V_n \times U^*$.
Hence $a \in V \cap \der^{-1}_1(V_1)\cap\cdots\cap\der^{-1}_n(V_n)$ and $a \notin P$ as $\der(a) \notin U$.
\end{proof}

\begin{corollary}\label{cor:loc bound}
Suppose that $I$ is $K$-linearly independent and $F = \bigcap_{\der \in I} \mathrm{Cons}(\der)$ is dense.
Then $\uptau_I$ is locally bounded if and only if $I$ is finite.
\end{corollary}

This gives the first example of a gt-henselian topology that is not locally bounded.
For example if $I = (\der_i)_{i \in \N}$ is an $\R$-linearly independent sequence of derivations $\R \to \R$ then $\uptau_I$ is a non-locally bounded gt-henselian topology on $\R$.

\begin{proof}
The right to left direction follows by Lemma~\ref{lem:loc bounded}.
Suppose that $I$ is infinite and $\uptau_I$ is locally bounded.
Let $U$ be a bounded open neighborhood of $0$.
Then there is a $\uptau$-open neighborhood $V$ of $0$ and $\der_1,\ldots,\der_n \in I$ such that $V \cap \der^{-1}_1(V) \cap \cdots \cap \der^{-1}_n(V)$ is contained in $U$.
So we may suppose that $U = V \cap \der^{-1}_1(V) \cap \cdots \cap \der^{-1}_n(V)$.
Let $J = \{\der_1,\ldots,\der_n\}$, so $U$ is $\uptau_J$-open.
Now $\{aU + b : a \in K^\times, b \in K\}$ is a basis for both $\uptau_I$ and $\uptau_J$, so $\uptau_I = \uptau_J$.
This is a contradiction by Lemma~\ref{lem:ind-family}.
\end{proof}

\begin{lemma}\label{lem:gt wt}
If there is a gt-henselian topology on $K$ then there is a gt-henselian topology with weight at most $|K|$.
\end{lemma}

One can prove this via an easy closure argument.
The required closure argument can be seen as a special case of the proof of the \lowenheim theorem, so we just apply that.

\begin{proof}
Recall that a collection $\mathcal{B}$ of subsets of $K$ is a neighborhood basis at $0$ for a gt-henselian topology if and only if we have the following~\cite[Fact~2.1]{lagth}.
\begin{enumerate}[leftmargin=*]
\item Every element of $\mathcal{B}$ contains $0$, $\{0\} \notin \mathcal{B}$, and some element of $\mathcal{B}$ does not contain $1$.
\item For any $V,V^* \in \mathcal{B}$ there is $U \in \mathcal{B}$ such that $U \subseteq V \cap V^*$.
\item For any $V \in \mathcal{B}$ there is $U \in \mathcal{B}$ such that $U - U$, $UU\subseteq V$ and $(1 + U)^{-1}\subseteq(1 + V)^{-1}$.
\item For any $\lambda \in K^\times$ and $V \in \mathcal{B}$ we have $\lambda U \subseteq V$ for some $U \in \mathcal{B}$.
\item For any $d \in \N$ and $V \in \mathcal{B}$ there is $U \in \mathcal{B}$ such that $x^{d + 2} + x^{d + 1} + a_d x^d + \cdots + a_1 x + a_0$ has a simple root in $V - 1$ when $a_0,\ldots,a_d \in U$.
\end{enumerate}
Suppose that $\uptau$ is gt-henselian and fix a neighborhood basis $\mathcal{E}$ for $\uptau$ at $0$.
Consider the two-sorted structure $(K,\mathcal{E})$ with sorts $K$ and $\mathcal{E}$, the field structure on $K$, and $\in$ as a binary relation between $K$ and $\mathcal{E}$.
By the \lowenheim theorem there is a subset $\mathcal{E}^* \subseteq \mathcal{E}$ of cardinality at most $|K|$ such that $(K,\mathcal{E}^*)$ is an elementary substructure of $(K,\mathcal{E})$.
In particular $\mathcal{E}^*$ satisfies (1)-(5) above and hence $\mathcal{E}^*$ is a neighborhood basis at $0$ for a gt-henselian topology $\uptau^*$ on $K$.
Finally, $\uptau^*$ has weight at most $|K|$ as $|\mathcal{E}|^*\le |K|$.
\end{proof}

\begin{fact}\label{fact:top-card}
If $\uptau$ is not discrete, $U$ is a nonempty $\uptau$-open subset of $K$, and $F$ is a subfield of $K$ over which $K$ has transcendence degree at least two, then there are $a,b \in U$ which are algebraically independent over $F$.
\end{fact}

Fact~\ref{fact:top-card} is a special case of \cite[Prop.~C]{topological_proofs} and has probably been proven in other places.

\begin{lemma}\label{lem:dense}
Suppose that $K$ has infinite transcendence degree and $\uptau$ has weight $\le \kappa = |K|$.
Then there are disjoint dense subsets $B,B^*$ of $K$ such that $B \cup B^*$ is algebraically independent and $|B| = \kappa = |B^*|$
\end{lemma}

\begin{proof}
Note that $K$ has transcendence degree $\kappa$.
We construct disjoint sequences $(b_i)_{i < \kappa}$, $(b^*_i)_{i < \kappa}$ of elements of $K$ such that the $b_i,b^*_i$ are algebraically independent.
Let $(U_i)_{i < \kappa}$ be a basis for $\uptau$.
Applying Fact~\ref{fact:top-card} let $b_0, b^*_0$ be elements of $U_0$ that are algebraically independent over $\Q$.
Suppose we have $(b_i)_{i <\lambda}$ and $(b^*_i)_{i < \lambda}$ for some $1\le \lambda < \kappa$.
Let $F$ be the algebraic closure of the field generated by the $b_i$ and $b^*_i$.
Then $F$ has transcendence degree $<\kappa$.
Again applying Fact~\ref{fact:top-card} take $b_{\lambda}, b^*_{\lambda}$ to be elements of $U_\lambda$  algebraically independent over $F$.
\end{proof}

\begin{fact}\label{fact:der}
Suppose that $F$ is a subfield of $K$, $B$ is a transcendence basis for $K$ over $F$, $F \to K$ is a derivation, and $B \to K$ is an arbitrary function.
Then there is a unique derivation $K \to K$ which extends both $F \to K$ and $B \to K$.
\end{fact}

Fact~\ref{fact:der} is a basic algebraic fact, see \cite[Cor.~1.9.4]{trans}.

\begin{lemma}\label{lem:get a lot}
Suppose that $K$ has infinite transcendence degree and $\uptau$ has weight at most $\kappa = |K|$.
Then there is a $K$-linearly independent collection $D$ of $2^\kappa$ derivations $K \to K$ such that the intersection of the constant subfields of the $\der \in D$ is dense in $K$.
\end{lemma}

\begin{proof}
By Lemma~\ref{lem:dense} there are disjoint dense algebraically independent subsets $B, B^*$ of $K$ with $|B| = \kappa = |B^*|$.
After possibly expanding $B^*$ we may suppose that $B \cup B^*$ is a transcendence basis for $K$.
Let $F$ be the subfield of $K$ generated by $B$, so $F$ is a dense subfield of $K$ and $B^*$ is a transcendence basis for $K/F$.
By Fact~\ref{fact:der} any function $B^* \to K$ uniquely extends to a derivation $K \to K$ that vanishes on $F$.
Now $K^{B^*}$, equipped with the pointwise addition, is a $K$-vector space of dimension $2^\kappa$, so there is a linearly independent subset $E$ of cardinality $2^\kappa$.
Apply Fact~\ref{fact:der} to extend each element of $E$ to a derivation $K \to K$ which vanishes on $F$ and note that the resulting family of derivations is linearly independent.
\end{proof}

\begin{fact}\label{fact:sperner}
If $X$ is a set of cardinality $\kappa \ge \aleph_0$ then there is a collection of $2^\kappa$ subsets of $X$ that are pairwise incomparable under containment.
\end{fact}

Fact~\ref{fact:sperner} is a basic combinatorial fact that follows, e.g., by noting that the graphs of a pair of distinct functions $X \to \{0,1\}$ are incomparable subsets of $X \times \{0,1\}$. 

\begin{proposition}\label{prop:many many tops}
Suppose that $K$ has infinite transcendence degree and $\uptau$ is gt-henselian.
Let $\kappa = |K|$ and suppose that $\uptau$ has weight at most $\kappa$.
Then there is a collection of $2^{2^\kappa}$ pairwise incomparable gt-henselian field topologies on $K$, each of which strictly refines $\uptau$.
\end{proposition}

\begin{proof}
By Lemma~\ref{lem:get a lot} there is a $K$-linearly independent collection $D$ of $2^\kappa$ derivations $K \to K$ and a $\uptau$-dense subfield $F$ of $K$ which is contained in the constant subfield of each $\der \in D$.
By Fact~\ref{fact:sperner} there is a family $E$ of $2^{2^\kappa}$ subsets of $D$ that are pairwise incomparable under containment.
By Lemma~\ref{lem:ind-family} the topologies $\uptau_J$ for $J \in E$ are pairwise incomparable and each $\uptau_J$ strictly refines $\uptau$.
Finally, each $\uptau_J$ is gt-henselian by Proposition~\ref{prop:key}.
\end{proof}

\begin{proposition}\label{prop:many tops}
Suppose that $K$ has infinite transcendence degree and $\uptau$ has weight $\eta$ for $\eta \le \kappa = |K|$.
Then there is a collection of $2^\kappa$ pairwise incomparable gt-henselian topologies on $K$, each of which has weight $\eta$ and strictly refines $\uptau$.
\end{proposition}

\begin{proof}
Let $E$ be as in the proof of Lemma~\ref{lem:get a lot} and let $J = (\uptau_\der)_{\der \in E}$.
Then $|J| = 2^{\kappa}$, by Lemma~\ref{lem:ind-family} the elements of $J$ are pairwise incomparable, by Proposition~\ref{prop:key} each element of $J$ is gt-henselian, and by Lemma~\ref{lem:der 0} we have $w(\uptau_\der) = \eta$ for all $\der \in E$.
\end{proof}

Proposition~\ref{prop:mm tops} follows by Propositions~\ref{prop:many tops}, \ref{prop:many many tops}, and Lemma~\ref{lem:gt wt}.

\begin{proposition}\label{prop:mm tops}
If $K$ has infinite transcendence degree and there is a gt-henselian topology on $K$ then there is a collection of $2^{2^{|K|}}$ pairwise incomparable gt-henselian topologies on $K$ and a collection of $2^{|K|}$ pairwise incomparable gt-henselian topologies on $K$ of weight $\le |K|$.
\end{proposition}

Corollary~\ref{cor:to will} follows from Proposition~\ref{prop:mm tops} and Johnson's theorem.

\begin{corollary}\label{cor:to will}
Suppose that $K$ is countable, large, and has infinite transcendence degree.
Then there is a collection of $2^{2^{\aleph_0}}$ pairwise incomparable gt-henselian topologies on $K$ and a collection of $2^{\aleph_0}$ pairwise incomparable second countable gt-henselian topologies on $K$.
\end{corollary}

Finally, Corollary~\ref{cor:loc bound} follows from Lemma~\ref{lem:loc bounded}, the fact that a locally bounded field topology on $K$ has weight at most $|K|$, and the proof of Proposition~\ref{prop:many tops}.

\begin{corollary}\label{cor:loc bound}
Suppose that $K$ has infinite transcendence degree, $\uptau$ is locally bounded, and $\kappa = |K|$.
Then there is a collection of $2^\kappa$ pairwise incomparable locally bounded gt-henselian topologies on $K$, each of which strictly refines $\uptau$.
\end{corollary}

\bibliographystyle{abbrv}
\bibliography{ref}
\end{document}